\documentclass{amsart}
\usepackage{graphicx}
\usepackage{amssymb}
\usepackage{amsmath}
\usepackage{amsfonts}

\newtheorem{theorem}{Theorem}[section]

\newtheorem{lemma}[theorem]{Lemma}

\newtheorem{fact}[theorem]{Fact}

\newcommand{\dt}{\mathcal{D}_n(T_2)}

\begin{document}


\title[]{Density and Fractal Property of the Class of Oriented Trees}


\author{Jan Hubi\v cka}
\address{Department of Applied Mathematics (KAM), Charles University,
Ma\-lo\-stransk\'e n\'am\v est\'i 25, Praha 1, Czech Republic}
\curraddr{}
\email{hubicka@kam.mff.cuni.cz}
\urladdr{}

\author{Jaroslav Ne\v set\v ril}
\address{Computer Science Institute of Charles University (IUUK),
Charles University, Ma\-lo\-stransk\'e n\'am\v est\'i 25, Praha 1, Czech
Republic}
\curraddr{}
\email{nesetril@iuuk.mff.cuni.cz}
\urladdr{}

\author{Pablo Oviedo}
\address{Departament de Matem\` atiques, Universitat Polit\` ecnica de Catalunya, Barcelona, Spain}
\curraddr{}
\email{pablo.oviedo@estudiant.upc.edu}
\urladdr{}



\thanks{The first author was supported by  project  18-13685Y  of  the  Czech  Science Foundation (GA\v CR) and by Charles University project Progres Q48.  The third author was supported by the Spanish Research Agency under project MTM2017-82166-P.}




{}

\begin{abstract}
We show a density theorem for the class of finite proper trees ordered by the homomorphism order, where a proper tree is an oriented tree which is not homomorphic to a path. We also show that every interval of proper trees, in addition to being  dense, is in fact universal. We end by considering the fractal property in the class of all finite digraphs.
This complements the characterization of finite dualities of finite digraphs.
\end{abstract}

\maketitle



\section{Introduction}
\label{section1}
In this note we consider finite directed graphs (or digraphs) and countable partial orders. A homomorphism between two digraphs $f:G_1\rightarrow G_2$ is an arc preserving mapping from $V(G_1)$ to $V(G_2)$. If such homomorphism exists we write $G_1\leq G_2$. The relation $\leq$, called the homomorphism order, defines a quasiorder on the class of all digraphs which, by considering equivalence classes, becomes a partial order. A core of a digraph is its minimal homomorphic equivalent subgraph. 

In the past three decades the richness of the homomorphism order of graphs and digraphs has been extensively studied \cite{llibre}. In 1982, Welzl showed that undirected graphs, with one exception, are dense \cite{origdens}. Later in 1996, Ne\v set\v ril and Zhu characterized the gaps and showed a density theorem for the class of finite oriented paths \cite{pathhomomorphism}. We contribute to this research by showing a density theorem for the class of oriented trees. We say that an oriented tree is \emph{proper} if its core is not a path. 

\begin{theorem}
\label{1}
Let $T_1$ and $T_2$ be two finite oriented trees satisfying $T_1<T_2$. If $T_2$ is a proper tree, then there exists a tree $T$ such that $T_1<T<T_2$.
\end{theorem}

This result was claimed by Miroslav Treml around 2005, but never published. Our proof of Theorem \ref{1} is new and simple, and leads to further consequences. In particular, we can show the following strengthening. Let us say that a partial order is universal if it contains every countable partial order as a suborder. 

\begin{theorem}
\label{2}
Let $T_1$ and $T_2$ be two finite oriented trees satisfying $T_1<T_2$. If $T_2$ is a proper tree, then the interval $[T_1,T_2]$ is universal. 
\end{theorem}

Recently, it has been shown that every interval in the homomorphism order of finite undirected graphs is either universal or a gap \cite{fractal}. As consequence of Theorem \ref{2}, this property, called fractal property, seems to be present in other classes of digraphs. In fact, we have shown the following result related to the class of finite digraphs.

\begin{theorem}
\label{3}
Let $G$ and $H$ be two finite digraphs satisfying $G<H$, where the core of $H$ is connected and contains a cycle. Then the interval $[G,H]$ is universal.
\end{theorem}

The proof of Theorem \ref{3} will appear in the full version of this note.

\section{Preliminaries}
\label{section2}
We follow the notation used in Hell and Ne\v set\v ril's book \cite{llibre}.

A \emph{digraph} $G$ is an ordered pair of sets $(V,A)$ where $V=V(G)$ is a set of elements called \emph{vertices} and $A=A(G)$ is a binary irreflexive relation on $V$. The elements $(u,v)$, denoted $uv$, of $A(G)$ are called \emph{arcs}. 

A \emph{path} is a digraph consisting in a sequence of different vertices $\{v_0,\dots,v_k\}$ together with a sequence of different arcs $\{e_1,\dots,e_k\}$ such that $e_i$ is an arc joining $v_{i-1}$ and $v_i$ for each $i=1,\dots,k$. A \emph{cycle} its defined analogously but with $v_0=v_k$. A \emph{tree} is a connected digraph containing no cycles. The \emph{height} of a tree is the maximum difference between forward and backward arcs of a subpath in it.

A \emph{homomorphism} from a digraph $G$ to a digraph $H$ is a mapping $f:V(G)\rightarrow V(H)$ such that $uv\in E(G)$ implies $f(u)f(v)\in E(H)$. It is denoted $f:G\rightarrow H$. If there exists a homomorphism from $G$ to $H$ we write $G\rightarrow H$, or equivalently, $G\leq H$. We shall write $G<H$ for $G\leq H$ and $H\nleq G$. The \emph{interval} $[G,H]$ consists in all digraphs $X$ such that $G\leq X\leq H$. A \emph{gap} is an interval in which there is no digraph $X$ such that $G<X<H$.

The relation $\leq$ is clearly a quasiorder which becomes a partial order by choosing a representative for each equivalence class, in our case the so called core. A \emph{core} of a digraph is its minimal homomorphic equivalent subgraph.

Given two partial orders $(\mathcal{P}_1,\leq_1)$ and $(\mathcal{P}_2,\leq_2)$, an \emph{embedding} from $(\mathcal{P}_1,\leq_1)$ to $(\mathcal{P}_2,\leq_2)$ is a mapping $\Phi:\mathcal{P}_1\rightarrow \mathcal{P}_2$ such that for every $a,b\in \mathcal{P}_1$, $a\leq b$ if and only if $\Phi(a)\leq \Phi(b)$. If such a mapping exists we say that $(\mathcal{P}_1,\leq_1)$ can be embedded into $(\mathcal{P}_2,\leq_2)$.

Finally, a partial order is \emph{universal} if every countable partial order can be embedded into it.

\section{Density Theorem}

In order to prove Theorem \ref{1} we shall construct a tree $\dt$ from a given proper tree $T_2$ which will satisfy $T_1<\dt<T_2$ for every tree $T_1<T_2$.

Given a tree $T$, a vertex $u\in V(T)$ and a set of vertices $S\subseteq V(T)$, the \emph{plank} from $u$ to $S$, denoted $P(u,S)$, is the subgraph induced by the vertices of every path which starts with $u$ and contains some vertex $v\in S$.

Let $T_2$ be the core of a proper tree. Then there exists a vertex $x\in V(T_2)$ such that $x$ is adjacent to at least three different vertices, name them $u,v,w$. Without loss of generality we shall assume that $ux$ and $wx$ are arcs. Let $X'\subseteq V(T_2)$ be the set of vertices, different from $u$ and $w$, which are adjacent to $x$. Note that $X'$ is not empty since $v\in X'$. Let $X=P(x,X')$, $U=P(x,\{u\})\backslash\{x\}$ and $W=P(x,\{w\})\backslash\{x\}$. Observe that $U\sqcup X\sqcup W\sqcup \{ux,wx\}=T_2$. See Figure~\ref{figure.T2}.

\begin{figure}
\centering
\begin{minipage}{0.4\textwidth}
\centering
\includegraphics[scale=0.5]{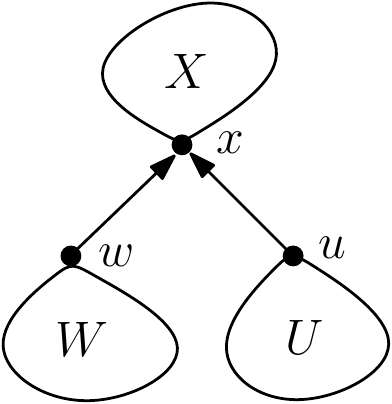}
\caption{Tree $T_2$.}
\label{figure.T2}
\end{minipage}%
\begin{minipage}{0.6\textwidth}
\centering
\includegraphics[scale=0.5]{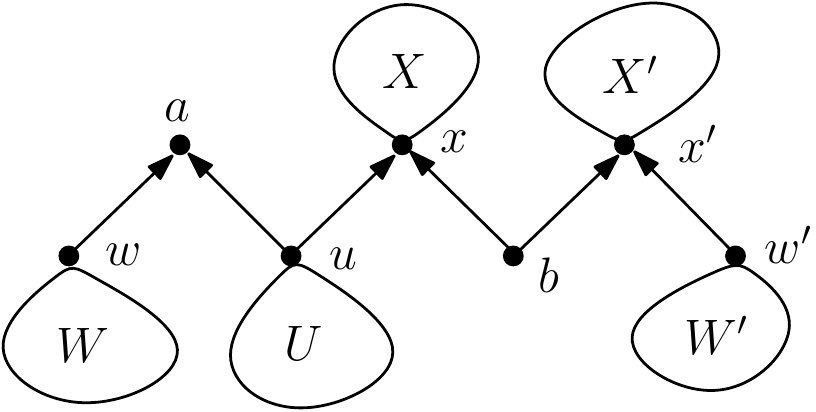}
\caption{Tree $\mathcal{D}_1(T)$.}
\label{figure.D1}
\end{minipage}
\end{figure}

\begin{figure}
\centering
\includegraphics[width=\textwidth]{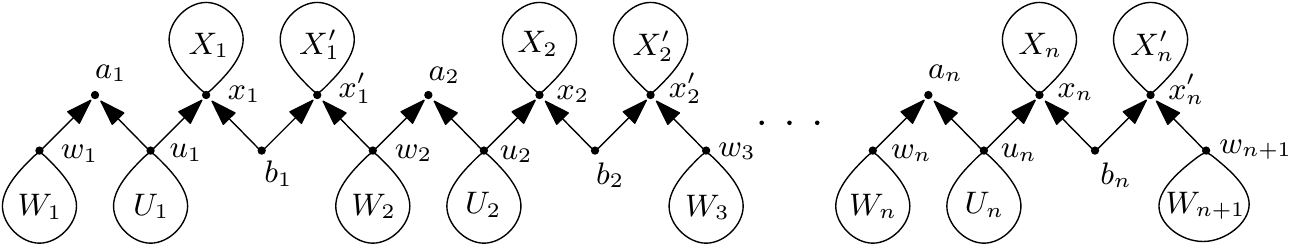}
\caption{Tree $\dt$. Observe the enumeration of the vertices and planks of each tree $\mathcal{D}_1(T)$.}
\label{figure.Dn}
\end{figure}

Now, let $\mathcal{D}_1(T_2)$ be the tree from Figure~\ref{figure.D1}, where $W$ and $W'$ are copies of the plank $W\subset T_2$, $U$ is a copy of $U\subset T_2$, and $X$ and $X'$ are copies of $X\subset T_2$.

Finally, let $\mathcal{D}_n(T_2)$ be a tree consisting in $n$ consecutive trees $\mathcal{D}_1(T_2)$ whose planks $W'$ are identified with the planks $W$ of the following trees. See Figure~\ref{figure.Dn}. We shall refer to the vertices $w_i,a_i,u_i,x_i,b_i,x'_i\in \dt$ for $i=1,\dots,n$ as \emph{labelled vertices}.

\begin{lemma}
\label{labelled.vertices}
Let $T_1$ and $T_2$ be finite oriented trees such that $T_2$ is a proper tree and $T_2\nrightarrow T_1$. If there exists a homomorphism $f:\dt\rightarrow T_1$, then every labelled vertex of $\dt$ is mapped to a different vertex of $T_1$.
\end{lemma}

\begin{proof}
Assume that $T_2$ is a core and consider a homomorphism $f:\dt\rightarrow T_1$. Observe that two consecutive labelled vertices can not be mapped via $f$ to the same vertex since it would imply that $T_1$ contains a loop. Now, observe that if any pair of labelled vertices of distance two are mapped to the same vertex, it will induce a homomorphism $T_2\rightarrow T_1$. This follows from the construction of $\dt$. See Figure~\ref{figure.Dn}. Finally, if two labelled vertices at distance greater or equal to three are mapped to the same vertex, it would imply that $T_1$ contains a cycle, which is a contradiction since $T_1$ is a tree. We conclude that every labelled vertex has to be mapped to a different vertex of $T_1$.
\end{proof}

A digraph $G$ is \emph{rigid} if it is a core and the only automorphism $f:G\rightarrow G$ is the identity. We shall use the following fact.

\begin{fact}
\label{rigid}
The core of a tree is rigid.
\end{fact}

\begin{proof}[Proof of Theorem \ref{1} (sketch)]
Assume that $T_2$ is a core. Let $n>|V(T_1)|$ and consider the tree $\dt$. It is clear that $\dt\rightarrow T_2$. It can also be checked that $T_2\nrightarrow \dt$ (here we might use Fact \ref{rigid}). To see that $\dt\nrightarrow T_1$ observe that by Lemma \ref{labelled.vertices} every labelled vertex in $\dt$ has to be mapped to a different vertex in $T_1$, but the number of labelled vertices in $\dt$ is greater than $|V(T_1)|$. Thus, $T_1<T_1+\dt<T_2$.

We end by joining $T_1$ with $\dt$ by a proper and long enough zig-zag. The method is similar to the one used in the proof of the density theorem for paths \cite{pathhomomorphism}.
\end{proof}

\section{Fractal property for proper trees}
\label{section3}
\begin{proof}[Proof of Theorem \ref{2} (sketch)]
Let $n>|V(T_1)|+2|V(T_2)|$ and consider the tree $\dt$. We know by Theorem \ref{1} that $T_1<T_1+\dt<T_2$.

Let $T$ be the core of $\dt$. By Lemma \ref{labelled.vertices} every labelled vertex in $\dt$ has to be mapped to a different vertex in $T$. Since $n>|V(T_1)|+2|V(T_2)|$, it follows that $T$ has at least $|V(T_1)|$ labelled vertex. Let $y$ and $z$ be the initial and ending labelled vertex of $T$ respectively. Let $T'$ be the tree obtained from $T$ by adding two new vertices $y'$ and $z'$ and joining $y'$ to $y$ and $z'$ to $z$ by a proper zig-zag of length 5 or 6 so $y'$ and $z'$ have the same level, as shown in Figure~\ref{figure.gadget}. Finally let $T''$ be the tree obtained by joining $T_1$ with $T'$ by a proper and long enough zig-zag.

\begin{figure}
\centering
\includegraphics[width=\textwidth]{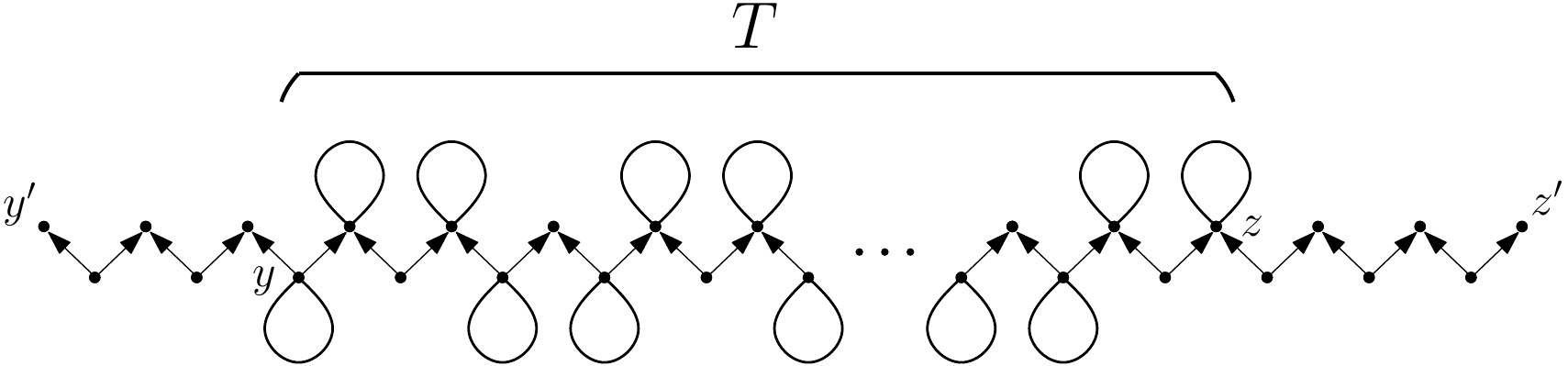}
\caption{This is an example of how $T'$ might look. The vertices $y$ and $z$ might be different from the ones in the figure but they must be labelled vertices of $\dt$.}
\label{figure.gadget}
\end{figure}

Now, we shall construct an embedding $\Phi$ from the homomorphism order of the class of oriented paths, which we know is a universal partial order \cite{univepaths}, into the interval $[T_1,T_2]$.

Given an oriented path $P$, let $\Phi(P)$ be the tree obtained by replacing each arc $v_1v_2$ in $P$ by a copy of $T''$ identifying $v_1$ with $y'$ and $v_2$ with $z'$. Observe that $T_1<\Phi(P)<T_2$. It is clear that any homomorphism $f:P_1\rightarrow P_2$ induces a homomorphism $g:\Phi(P_1)\rightarrow \Phi(P_2)$ by identifying arcs with copies of $T''$. To see the opposite, observe that since $T$ is rigid by Fact \ref{rigid}, every copy of $T$ in $\Phi(P_1)$ must be map via the identity to some copy of $T$ in $\Phi(P_2)$. It follows that adjacent copies of $T''$ in $\Phi(P_1)$ must be mapped to adjacent copies of $T''$ in $\Phi(P_2)$. Hence, each homomorphism $g:\Phi(P_1)\rightarrow \Phi(P_2)$ induces a homomorphism $f:P_1\rightarrow P_2$.
\end{proof}

\section{Fractal property for finite digraphs}
\label{section4}
We say that a class of digraphs $\Vec{\mathcal{G}}$ has the \emph{fractal property} if every interval in the homomorphism order $(\Vec{\mathcal{G}},\leq)$ is either universal or a gap. The fractal property was introduced by Ne\v set\v ril \cite{nesetril} and it has been shown recently that the class of finite undirected graphs (or symmetric digraphs) has the fractal property \cite{fractal}. In this note, we have shown that the class of proper trees has also the fractal property (as consequence of Theorem \ref{2}). However, the class of finite digraphs, and even the class of oriented trees, is more complicated.

Ne\v set\v ril and Tardif characterised all gaps in the homomorphism order of finite digraphs \cite{dual}. It was shown that for every tree $T$ there exists a digraph $G$ such that $[G,T]$ is a gap, and that all gaps have this form. Theorem \ref{1} contributes to this result by implying that if $[G,T]$ is a gap and $T$ is a proper tree, then $G$ must contain a cycle.

The characterisation of universal intervals in the homomorphism order of finite digraphs seems to be complicated. Related to this issue, we have stated Theorem \ref{3}. Its proof combines some techniques already used \cite{llibre,fractal} with some extra arguments, and will appear in the full version of this note. This result together with Theorem \ref{2} imply that the class of finite digraphs whose cores are not paths has the fractal property. However, intervals of the form $[G,P]$ where the core of $P$ is a path remain to be studied and characterised. 



\end{document}